\documentclass[12pt]{amsart}
\usepackage{amsmath,amssymb,amsbsy,amsfonts,latexsym,amsopn,amstext,
                                               amsxtra,euscript,amscd,bm}
\usepackage[margin=1in]{geometry}                           
\usepackage{url}
\usepackage[colorlinks,linkcolor=blue,anchorcolor=blue,citecolor=blue]{hyperref}
\usepackage{color}

\begin{document}
\newtheorem{problem}{Problem}
\newtheorem{theorem}{Theorem}
\newtheorem{lemma}[theorem]{Lemma}
\newtheorem{claim}[theorem]{Claim}
\newtheorem{cor}[theorem]{Corollary}
\newtheorem{prop}[theorem]{Proposition}
\newtheorem{definition}{Definition}
\newtheorem{question}[theorem]{Question}

%%%%%%%%%%%%%%%%%%%%%%%
% Alphabet blackboard %
%%%%%%%%%%%%%%%%%%%%%%%

\def\F{{\mathbb F}}

%%%%%%%%%%%%%%%%%%
%% PAPER BEGINS %%
%%%%%%%%%%%%%%%%%%

\title[Gauss sums in $\F_q$]{Improved bounds on Gauss sums in arbitrary finite fields}
\author{Ali Mohammadi}
\address{School of Mathematics and Statistics, University of
Sydney, NSW, 2006, Australia}
\email{alim@maths.usyd.edu.au}

%\date{Today}

\pagenumbering{arabic}

\begin{abstract} 
Let $q$ be a power of a prime and let $\F_q$ be the finite field consisting of $q$ elements.  We establish new explicit estimates on Gauss sums of the form $S_n(a) = \sum_{x\in \F_q}\psi_a(x^n)$, where $\psi_a$ is a nontrivial additive character. In particular, we show that one has a nontrivial upper bound on $|S_n(a)|$ for certain values of $n$ of order up to $q^{1/2 + 1/68}$. Our results improve on the previous best known bound, due to Zhelezov.
\end{abstract}
\maketitle

\section{Introduction}

For a prime $p$, let $\F_q$ denote the finite field with $q = p^m$ elements and let $\F_q^{*} =\F_q \backslash \{0\}$.  We define the trace function $$Tr( x) = \sum_{i = 0}^{m-1} x^{p^i}.$$
Let $e_p(x) = \exp(2\pi i x/p)$ and $\psi(x) = e_p(Tr(x))$. Then, for $a \in \F_q$, the functions $\psi_a(x) = \psi(ax)$ represent the additive characters of $\F_q$, with $\psi_{0}$ being the trivial character.
We define a Gauss sum as 
\[
 S_n(a) = \sum_{x\in \F_q}\psi_a(x^n).
\]
It is easy to verify that if $d = \gcd(n, q -1)$, then $S_n(a) = S_d(a)$. Hence, throughout the paper we shall assume that $n | (q-1)$. Let $H$ be a multiplicative subgroup of $\F_q^*$. We consider sums of the form
\begin{equation}
 \label{SaH}
S(a, H) = \sum_{h\in H}\psi_a(h).
\end{equation}
Since $\F_q^*$ is a cyclic group, any subgroup of $\F_q^*$ coincides with the group
\begin{equation}
\label{eqn:nthpowers}
\{x^n: x \in \F_q^{*} \}
\end{equation}
for some $n$. Indeed, fixing $H$ to denote the group of $n$th powers~\eqref{eqn:nthpowers}, it follows that 
\begin{equation}
\label{eqn:EStoGS}
S_n(a) = 1 + nS(a, H).
\end{equation}

By a classical result of Weil we have the bound
\begin{equation}
\label{eqn:WeilBound}
\max_{a\in\F_q^{*}} |S_n(a)| < (n - 1)q^{1/2}.
\end{equation}
This bound becomes trivial for $n \geq q^{1/2} + 1$. In this range the first nontrivial estimate was obtained by Shparlinski~\cite{Shpar} for values of $n$ of order up to $q^{1/2}p^{1/6}$. Bourgain and Chang~\cite{BouCha} showed that for any $\epsilon > 0$, there exists some $\delta = \delta(\epsilon) > 0$ such that if $n$ satisfies
\begin{equation}
\label{eqn:BCCondition}
\gcd\bigg(n, \frac{q -1}{p^\nu - 1} \bigg) <\frac{q^{1-\epsilon}}{p^{\nu}}
\end{equation}
for all $\nu$ with $1\leq \nu <m$, $\nu|m$, then
\[
\max_{a\in\F_q^{*}} |S_n(a)| \ll q^{1-\delta}.
\]
Later, Bourgain and Glibichuk~\cite{BouGlib} obtained an explicit estimate of the form $|S(a, H)| = o(|H|)$, which holds for subgroups $H$, under the restriction that $|H\cap G| < |H|^{1-\eta}$ for some $\eta > 0$ and all proper subfields $G$. Such a restriction is clearly needed as to ensure that $H$ does not largely correlate with a subfield of $\F_q$.

Note that proper subfields $G$ of $\F_q$ have cardinality $|G| = p^{\nu}$, where $1 \leq \nu < m$ with $\nu | m.$ Also recall the group of $n$th powers $H$ is of order $(q -1)/n$. Furthermore, since $H$ and $G^{*}$ are cyclic groups, so is their intersection $H \cap G^{*}$, so that
\begin{equation}
\label{eqn:subfieldintersection}
|H\cap G| =  \frac{\gcd(n,\frac{q-1}{p^{\nu}-1})(p^{\nu}-1)}{n}.
\end{equation}
Hence, it is clear that a condition such as~\eqref{eqn:BCCondition} is also necessary.

For any sets $A, B\subset \F_q$, we define the additive energy between $A$ and $B$ as the quantity
\begin{equation}
\label{eqn:EnergyDefn}
E_{+}(A, B) = |\{ (a_1, a_2, b_1, b_2) \in A^{2}\times B^{2} : a_1 + b_1 = a_2 + b_2\}|
\end{equation}
and write simply $E_{+}(A)$ instead of $E_{+}(A, A)$. We have the trivial upper bound $E_{+}(A) \leq |A|^{3}$, which becomes an equality if $A = G$ for any subfield $G$ of $\F_q$. Given a multiplicative subgroup $H$, it is well known (see~\cite[Lemma 3.1]{KonShp}) that a nontrivial estimate on $E_{+}(H)$ leads to a nontrivial estimate for $|S(a, H)|$.
Zhelezov~\cite{Zhel} showed that if $|H| \ll q^{1/2}$ and
\begin{equation}
\label{eqn:ZhelEnergyCondition}
|H \cap G| \ll |H|^{1 - \delta_1}
\end{equation}
for all proper subfields $G$, with $\delta_1 = 119/605$, then
\begin{equation}
\label{eqn:Zhelenergy}
E_{+}(H) \ll |H|^{3 - \delta_2},
\end{equation}
where $\delta_2 = 1/56 - o(1).$ This was then used to obtain the Gauss sum estimate
\begin{equation}
\label{eqn:ZhelGSbound}
\max_{a\in\F_q^{*}} |S_n(a)| \ll q^{(7-2\delta_{2})/8}n^{(1+\delta_{2})/4},
\end{equation}
assuming $n$ satisfies the condition
\begin{equation}
\label{eqn:zhelgcdcondition}
\gcd\bigg(n, \frac{q -1}{p^\nu - 1} \bigg) \ll \frac{n^{\delta_1}q^{1-\delta_1}}{p^{\nu}}
\end{equation}
for all $\nu$ with $ 1\leq \nu < m$, $\nu | m$. 

In this paper, following the same approach as \cite[Lemma 3]{Zhel}, we improve the bound~\eqref{eqn:Zhelenergy} and remove the restriction $|H| \ll q^{1/2}$. As outlined in \cite[Remark~1]{Zhel}, there is an interplay between the restriction \eqref{eqn:ZhelEnergyCondition} and the bound \eqref{eqn:Zhelenergy}. We relax this dependence considerably and make it explicit in the presentation of our result. Furthermore, our estimates also hold for groups $H$ which satisfy $|H\cap G| \ll |G|^{1/2}$, for all proper subfields $G$. These results in turn lead to improvements and generalisations of the bound~\eqref{eqn:ZhelGSbound}.

We also consider a similar approach as Bourgain and Garaev, in~\cite[Theorem~1.2]{BouGar}, to give an explicit bound on trilinear exponential sums 
\begin{equation}
\label{eqn:IntroTrilinear}
\sum_{x\in X}\sum_{y\in Y} \sum_{z \in Z} \psi(xyz),
\end{equation}
for sets $X, Y, Z \subset \F_q$. This bound is then used to estimate $S(a, H)$. 

In the context of prime fields, progress towards obtaining explicit estimates of Gauss sums for large values of $n$ can be traced back to Shparlinski~\cite{Shpar1}. This was improved by Heath-Brown and Konyagin~\cite{HeaKon} based on Stepanov's method. Much of the subsequent progress on this problem has been largely dependent on the development of the sum-product phenomenon in finite fields. In particular, the work of Bourgain and Garaev~\cite{BouGar} may be viewed as a significant step in this direction. We refer the interested reader to~\cite{RNRuShk, MRSS, Shkr} for some of the currently best-known related results. 

We mention that the sum-product type techniques and estimates that are known over $\F_p$ are far superior and do not seem to readily extend to the more general setting of $\F_q$. Nevertheless, this paper seeks to address some of the important questions in the study of exponential sums in $\F_q$ using sum-product results which are currently available in this setting.

\subsection*{General notation}
For a set $A\subset \F_q$, we define the difference set $A-A = \{ a - b : a, b \in A\}$ and the ratio set $A/A = \{ a / b : a, b \in A, b\not=0\}$.

Given positive real numbers $X$ and $Y$, we use $X = O(Y)$, $ X \ll Y$ or $Y\gg X$ to imply that there exists an absolute constant $c >0$ such that $X \leq cY$.
If the constant $c$ depends on a parameter $\epsilon$, we write $X \ll_{\epsilon} Y$. If $X \ll Y$ and $Y \ll X$, we write $X \approx Y$. We use $X \lesssim Y$ to mean that there exists an absolute constant $c >0$ such that $X \ll (\log Y)^{c} Y$.

\section{Main Results}
We have the following improvement of~\eqref{eqn:Zhelenergy}.
\begin{theorem}
\label{thm:AddEnergy}
Let $A$ be any subset of $\F_q$, satisfying $|A/A| \approx |A|$. Fix $\delta \leq 1/33$. There exists an absolute constant $\lambda > 0$ such that if 
\begin{equation}
\label{eqn:energycnd1}
 |A \cap cG| \leq \lambda |A|^{1 - \delta} 
\end{equation}
for all proper subfields $G$ of $\F_q$ and elements $c\in \F_q$, then
\begin{equation}
\label{eqn:EnergyBound}
E_{+}(A) < \max \bigg\{ |A|^{3-\delta}, \frac{|A|^{3 + 1/33}}{q^{1/33}} \bigg\}.
\end{equation}
If, for all proper subfields $G$ and elements $c$ in $\F_q$, the set $A$ satisfies
\begin{equation}
\label{eqn:energycnd2}
 |A \cap cG| \ll |G|^{1/2}, 
\end{equation}
then estimate~\eqref{eqn:EnergyBound} holds with $\delta = 1/33$.
\end{theorem}

As a consequence of Theorem~\ref{thm:AddEnergy}, we obtain the following estimates on Gauss sums.
\begin{theorem}
\label{thm:GaussSumViaEnergy}
Let $n$ denote an integer such that $n | (q-1)$ and fix $\delta \leq 1/33$. There exists an absolute constant $\lambda > 0$ such that if $n$ satisfies
\begin{equation}
\label{eqn:gausssumcnd1}
\gcd \bigg(n, \frac{q -1}{p^\nu - 1}\bigg)  \leq  \lambda \frac{n^{\delta}q^{1-\delta}}{p^{\nu}}
\end{equation}
for all $\nu$ with $1\leq \nu < m$, $\nu | m$, then we have the estimates
\begin{equation}
\label{eqn:GSBound1}
\max_{a\in\F_q^{*}} |S_n(a)| \ll q^{(3-\delta)/4}n^{(2+\delta)/4} + q^{3/4}n^{65/132}
\end{equation}
and
\begin{equation}
\label{eqn:GSBound2}
\max_{a\in\F_q^{*}} |S_n(a)| \ll q^{(7-2\delta)/8}n^{(1+\delta)/4} + q^{7/8}n^{8/33}.
\end{equation}
Moreover, estimates~\eqref{eqn:GSBound1} and~\eqref{eqn:GSBound2} hold with $\delta = 1/33$ if $n$ satisfies
\begin{equation}
\label{eqn:gausssumcnd2}
\gcd \bigg(n, \frac{q -1}{p^\nu - 1}\bigg)  \ll  \frac{n}{p^{\nu/2}}
\end{equation}
for all $\nu$ with $1\leq \nu < m$, $\nu | m$.
\end{theorem}

Estimate \eqref{eqn:GSBound2} is stronger than \eqref{eqn:GSBound1} for $n > q^{1/2}$. Let $\delta = 1/33$. Then, for any fixed sufficiently small $\epsilon >0$, \eqref{eqn:GSBound2} is nontrivial for any $n \ll q^{1/2 +1/68 -\epsilon}$ which satisfies either of the restrictions \eqref{eqn:gausssumcnd1} or \eqref{eqn:gausssumcnd2}. In this range, \eqref{eqn:GSBound2} also improves on \eqref{eqn:ZhelGSbound}. It is also worth noting that for $n\leq q^{1/2}$, estimate \eqref{eqn:GSBound1} improves on the classical Weil bound \eqref{eqn:WeilBound} in the range $n \gg q^{1/2 -1/134}$.

\begin{theorem}
\label{thm:IncSums}
Let $g\in\F_q^*$ be of multiplicative order $t$. Suppose that $t$ satisfies
\begin{equation}
\label{eqn:IncSumCon}
\gcd(t, p^\nu - 1) \ll p^{\nu/2}
\end{equation}
for all $\nu$ with $1\leq \nu < m$, $\nu | m$. Then, for all $K \leq t$, we have the estimates
\begin{equation}
\label{eqn:IncPars}
\sum_{a\in \F_q}\Big|\sum_{j\leq K}\psi_a(g^j)\Big|^{4} \ll qK^{3-1/33} + q^{1-1/33}K^{3+1/33},
\end{equation}
\begin{equation}
\label{eqn:IncSum1}
\max_{a\in\F_q^{*}}\Big|\sum_{j\leq K}\psi_a(g^j)\Big| \ll q^{1/4}K^{65/132} + q^{8/33}K^{67/132},
\end{equation}
\begin{equation}
\label{eqn:IncSum2}
\max_{a\in\F_q^{*}}\Big|\sum_{j\leq K}\psi_a(g^j)\Big| \ll q^{1/8}K^{49/66} + q^{31/264}K^{25/33}.
\end{equation}
\end{theorem}

Our next result on trilinear exponential sums extends the arguments of \cite[Theorem~1.2]{BouGar} to arbitrary finite fields.
\begin{theorem}
\label{thm:TrilinearSumsViaEnergy}
Let $X, Y, Z \subset \F_q$ with $|X| \geq |Y| \geq |Z|$ and suppose that for all proper subfields $G$ of $\F_q$ and elements $c \in \F_q$, we have
\begin{equation}
\label{eqn:trilinearantifield}
|X \cap cG| \ll |G|^{1/2}.
\end{equation}
For any complex valued weights $(\alpha_x)_{x\in X}$, $(\beta_y)_{y\in Y}$ and $(\gamma_z)_{z\in Z}$, with 
\begin{equation}
\label{eqn:WeightsC}
\max_{x\in X}|\alpha_x| \leq 1,\quad \max_{y\in Y}|\beta_y| \leq 1,\quad \max_{z\in Z}|\gamma_z| \leq 1,
\end{equation}
we have
\begin{align}
\label{eqn:TrilinearB}
\bigg|\sum_{x\in X}\sum_{y\in Y} \sum_{z \in Z}\alpha_{x}\beta_{y}\gamma_{z}\psi(xyz)\bigg| 
&\ll q^{9/32 + o(1)} |X|^{95/128} |Y|^{3/4}|Z|^{15/16} \nonumber
\\ &+ q^{17/58 + o(1)} |X|^{43/58} |Y|^{87/116}|Z|^{53/58} 
\\ &+ q^{35/128 + o(1)} |X|^{97/128} |Y|^{3/4}|Z|^{15/16}. \nonumber
\end{align}
\end{theorem}

\begin{cor}
\label{cor:TrilinearEqSizes}
Let $X, Y, Z \subset \F_q$ with $|X| = |Y| = |Z|$ and suppose that for all proper subfields $G$ of $\F_q$ and elements $c \in \F_q$, $X$ satisfies \eqref{eqn:trilinearantifield}. For any complex valued weights $(\alpha_x)_{x\in X}$, $(\beta_y)_{y\in Y}$ and $(\gamma_z)_{z\in Z}$ satisfying \eqref{eqn:WeightsC} we have
\[
\bigg|\sum_{x\in X}\sum_{y\in Y} \sum_{z \in Z}\alpha_{x}\beta_{y}\gamma_{z}\psi(xyz)\bigg|  \ll q^{9/32 + o(1)}\big(|X|^{311/128} + q^{-1/128}|X|^{313/128}\big).
\]
\end{cor}

Applying Corollary~\ref{cor:TrilinearEqSizes}, one obtains the following estimate of Gauss sums.
\begin{theorem}
\label{thm:GaussSumViaTrilinearSum}
For any $n|(q-1)$ satisfying~\eqref{eqn:gausssumcnd2} we have
\[
\max_{a\in\F_q^{*}} |S_n(a)| \ll q^{91/128 + o(1)}n^{73/128} + q^{92/128+o(1)}n^{71/128}.
\]
\end{theorem}

\begin{cor}
Combining Theorem~\ref{thm:GaussSumViaEnergy} and Theorem~\ref{thm:GaussSumViaTrilinearSum}, for any $n|(q-1)$ satisfying~\eqref{eqn:gausssumcnd2}, and any sufficiently small $\epsilon >0$, we have
\begin{align*}
\label{combined}
\max_{a\in \F_q^{*}} |S_n(a)| \ll  \begin{cases}q^{1/2}n,  &\text{if} \quad  n \lesssim q^{1/2 - 1/130}, \\  q^{92/128+o(1)}n^{71/128}, &\text{if} \quad q^{1/2 - 1/130} \lesssim n \ll q^{1/2}, \\ q^{91/128 + o(1)}n^{73/128},  &\text{if} \quad q^{1/2}\ll n \lesssim q^{1/2 + 1/2642}, \\ q^{229/264}n^{17/66}, &\text{if} \quad q^{1/2 + 1/2642} \lesssim n \ll q^{1/2 + 1/68-\epsilon}.\end{cases}
\end{align*}
\end{cor}
\section{Preparations}
We require a sum-product type estimate due to Roche-Newton~\cite{Roche}. The following version can be found in \cite[Lemma~8]{Moh}.
\begin{lemma}
\label{lem:sum-ratio}
Let $A \subseteq \F_q$ and let $0 < \eta < 1/8$. Suppose $|A| \ll q^{1/2}$ and that
\begin{equation}
\label{eqn:SumRatioCondition}
 |A \cap cG| \leq \max \big\{C|G|^{1/2},\eta|A|\big\} 
\end{equation}
for all proper subfields $G$ of $\F_q$, elements $c\in \F_q$ and some constant $C > 0$. Then either
\[ 
|A - A|^7|A/A|^4 \gg_{\eta} |A|^{12} \qquad \text{or} \qquad |A - A|^6|A/A|^5 \gg_{\eta} |A|^{12}.
\]
If $|A| > \eta^{-1}q^{1/2}$, then
\[
 |A - A|^7|A/A|^4 \gg_{\eta} |A|^{10}q.
\]
\end{lemma}

For sets $A, B \subset \F_q$, we define the multiplicative energy $E_{\times}(A, B)$ as the multiplicative analogue of~\eqref{eqn:EnergyDefn} and denote $E_{\times}(A)= E_{\times}(A, A)$. It follows, from an application of the Cauchy-Schwarz inequality, that
\begin{equation}
\label{eqn:MECS}
E_{\times}(A, B) \leq E_{\times}(A)^{1/2}E_{\times}(B)^{1/2}.
\end{equation}
Our next lemma, which appears in \cite[Lemma~10]{Moh}, is a basic extension of \cite[Theorem~1.4]{RoLi}.
\begin{lemma}
\label{lem:energybound}
Let $A \subseteq \F_q$. Suppose that, for all elements $c$ and proper subfields $G$ in $\F_q$, we have $|A \cap cG| \ll |G|^{1/2}.$ Then we have the estimate
\begin{align}
\label{eqn:MEB}
E_{\times}(A) \ll \log|A| \cdot \max\{|A - A|^{7/4}|A|, |A - A|^{6/5}|A|^{8/5}, |A - A|^{7/4}|A|^{3/2}q^{-1/4}\}.
\end{align}
\end{lemma}

We recall two formulations of the Balog-Szemer\'{e}di-Gowers theorem.
Lemma~\ref{BSG}, below, is due to Bourgain and Garaev~\cite{BouGar}.
\begin{lemma}
\label{BSG}
Let $X,Y \subseteq \F_q$ and $G\subseteq X \times Y$. Then there exists $X^{'}\subseteq X$ with $|X^{'}| \gg |G|/|Y|$ such that
\[
|X^{'} - X^{'}| \ll \frac{|X \overset{G}{-} Y|^{4} |X|^4 |Y|^3}{|G|^5},
\]
where $X \overset{G}{-} Y = \{x - y: (x, y) \in G \}.$ 
\end{lemma}
The following variation is due to Schoen~\cite{Schoen}.
\begin{lemma}
\label{BSGSchoen}
Let $A\subset \F_q$. Suppose that $E_{+}(A) = \kappa |A|^3$. Then there exists $A^{'} \subseteq A$ such that $|A^{'}| \gg \kappa|A|$ and 
\[
|A^{'} - A^{'}| \ll \kappa^{-4}|A^{'}|.
\]
\end{lemma}

We require the following estimate on bilinear exponential sums. See \cite[Lemma 7]{BouGlib} for a proof.
\begin{lemma}
\label{lem:ClassicalBilinearSum}
For sets $X, Y \subseteq \F_q$, we have
\[
\bigg|\sum_{x\in X}\sum_{y\in Y}\alpha_{x}\beta_{y}\psi(xy)\bigg|\leq \sqrt{qNM},
\]
where $(\alpha_x)_{x\in X}$ and $(\beta_y)_{y\in Y}$ are any complex valued weights satisfying
\[
\sum_{x\in X} |\alpha_x|^{2} = N,\qquad \sum_{y\in Y} |\beta_y|^{2} = M.
\]
\end{lemma}

In the next lemma we recall two inequalities which are derived using Parseval's identity and different applications of H\"older's inequality. Also see \cite{BouGar, KonShp, RNRuShk}.
\begin{lemma}
\label{lem:BilinearHolder}
For sets $X, Y \subset \F_q$, we have the inequalities
\begin{equation}
\label{eqn:BH1}
\bigg|\sum_{x\in X}\sum_{y\in Y}\psi(xy)\bigg|^{4} \leq q\cdot \min\big\{|X|^{3}E_{+}(Y), |Y|^{3}E_{+}(X)\big\}
\end{equation}
and
\begin{equation}
\label{eqn:BH2}
\bigg|\sum_{x\in X}\sum_{y\in Y}\psi(xy)\bigg|^{8} \leq q|X|^{4}|Y|^{4}E_{+}(X)E_{+}(Y).
\end{equation}
\end{lemma}

We record a simple consequence of Lemma~\ref{lem:BilinearHolder}, which also appears in~\cite[equation (3.7)]{KonShp}.
\begin{lemma}
\label{lem:gausssumenergy}
Let $H$ be a multiplicative subgroup of $\F_q^*$ then 
\[
\max_{a \in \F_q^{*}}|S(a, H)| \leq \min\bigg\{\bigg(\frac{q E_{+}(H)}{|H|}\bigg)^{1/4}, q^{1/8}E_{+}(H)^{1/4}\bigg\}.
\]
\end{lemma}

Next, we state a standard dyadic pigeonholing argument.
\begin{lemma}
\label{lem:dyadic}
Let $X \subseteq \F_q$ and let $f$ be a function satisfying $0 \leq f(x) \leq M$ for all $x \in X$. Suppose that 
\[
\sum_{x\in X} f(x) \geq K.
\]
Then there exists a subset $X^{'} \subseteq X$ and a number $N\geq 1$ such that $N < f(x) \leq 2N$ for all $x \in X^{'}$ and 
\[
\sum_{x\in X^{'}} f(x) \geq \frac{K}{\log_2 M}.
\]
Consequently we have $N|X^{'}| \geq K/(2\log_2 M)$.
\end{lemma}

\section{Proof of Theorem \ref{thm:AddEnergy}}
Let $A\subset \F_q$ with $|A| = q^{\beta}$. First, assume $A$ satisfies restriction~\eqref{eqn:energycnd1} for some $\lambda > 0$ to be specified. Suppose, for a contradiction, that $E_{+}(A) = |A|^{3-\delta_{*}}$ for some $\delta_{*} < \min\{\delta, (\beta^{-1} -1)/33\}$. By Lemma~\ref{BSGSchoen}, with $\kappa = |A|^{-\delta_{*}}$, there exists a subset $A^{'} \subseteq A$, with
\begin{equation}
\label{eqn:Hprimesize}
|A^{'}| \gg |A|^{1-\delta_{*}},
\end{equation}
such that 
\[
|A^{'} - A^{'}| \ll |A|^{4\delta_{*}}|A^{'}| \ll |A^{'}|^{1 + \frac{4\delta_{*}}{1-\delta_{*}}}.
\]
Since $|A/A| \approx |A|$, we also have
\[
 |A^{'}/A^{'}| \ll |A| \ll |A^{'}|^{\frac{1}{1-\delta_{*}}}.
\]
We choose $\lambda$ to be sufficiently small so that, in view of~\eqref{eqn:Hprimesize}, we have
\[
\lambda |A|^{1-\delta_{*}} < \frac{|A^{'}|}{8}.
\]
It follows that for all elements $c$ and proper subfields $G$ in $\F_q$ 
\[
|A^{'} \cap cG| \leq |A \cap cG| \leq \lambda |A|^{1-\delta} < \lambda |A|^{1-\delta_{*}} < \frac{|A^{'}|}{8}.
\]
Hence we may apply Lemma~\ref{lem:sum-ratio} to the set $A^{'}$. If $|A^{'}| \ll q^{1/2}$, the lower bound for $\delta_{*}$ is determined by the equation
\[
7\bigg(1 + \frac{4\delta_{*}}{1-\delta_{*}}\bigg) + 4\bigg(\frac{1}{1-\delta_{*}}\bigg) \geq 12,
\]
which gives $\delta_{*} \geq 1/33$. If $|A^{'}| = q^{\beta_*}$ with $\beta_* > 1/2$, we have
\[
7\bigg(1 + \frac{4\delta_{*}}{1-\delta_{*}}\bigg) + 4\bigg(\frac{1}{1-\delta_{*}}\bigg) \geq 10 + \beta^{-1}_{*},
\]
which implies
\[
\delta_{*} \geq \frac{\beta^{-1}_{*} - 1}{31 + \beta^{-1}_{*}} > \frac{\beta^{-1} - 1}{33}.
\]
Thus, in either case, we have a contradiction on our choice of $\delta_{*}$. This concludes the proof of estimate \eqref{eqn:EnergyBound} under restriction~\eqref{eqn:energycnd1}.

Next, under restriction~\eqref{eqn:energycnd2}, for all proper subfields $G$ and elements $c$ in $\F_q$, we have
\[
|A^{'} \cap cG| \leq |A \cap cG| \ll |G|^{1/2},
\]
so that Lemma~\ref{lem:sum-ratio} may be applied to the set $A^{'}$. Hence, one may simply repeat the argument above, with $\delta = 1/33$.

\section{Proof of Theorem \ref{thm:GaussSumViaEnergy}}
We shall find useful the following observation.

\begin{claim}
\label{claim:EquivalencyofConditions}
Let $H$ denote a multiplicative subgroup of $\F^*_q$ and let $\lambda$ be the absolute constant given by Theorem~\ref{thm:AddEnergy}. Fix $\delta \leq 1/33$ and suppose that for all proper subfields $G$ of $\F_q$, we have
\begin{equation}
\label{eqn:Subgroupcnd1}
 |H \cap G| \leq \lambda |H|^{1 - \delta} .
\end{equation}
Then $H$ also satisfies condition~\eqref{eqn:energycnd1}. If $H$ satisfies
\begin{equation}
\label{eqn:Subgroupcnd2}
 |H \cap G| \ll |G|^{1/2},
\end{equation}
then it also satisfies condition~\eqref{eqn:energycnd2}.
\end{claim}

\begin{proof}
 Given an arbitrary proper subfield $G$ and an element $c$ in $\F_q$, suppose that the intersection $H \cap cG$ contains some element $d$, or else the desired conditions are trivially satisfied. Then, noting that $H = dH$ and $cG= dG $, we have $H \cap cG = d(H \cap G).$ Hence if conditions~\eqref{eqn:energycnd1} and~\eqref{eqn:energycnd2} fail, then so do conditions~\eqref{eqn:Subgroupcnd1} and~\eqref{eqn:Subgroupcnd2} respectively.
\end{proof}
\begin{proof}[Proof of Theorem \ref{thm:GaussSumViaEnergy}]
If a multiplicative group $H$ satisfies condition~\eqref{eqn:Subgroupcnd1} for any $\delta \leq 1/33$, by Claim~\ref{claim:EquivalencyofConditions}, we can apply Theorem~\ref{thm:AddEnergy}, so that by Lemma~\ref{lem:gausssumenergy}, it follows
\begin{equation}
\label{eqn:ExpSumBound1}
\max_{a\in\F_q^{*}}|S(a, H)| < q^{1/4} \cdot \max \bigg\{|H|^{(2-\delta)/4}, \frac{|H|^{67/132}}{q^{1/132}} \bigg\}
\end{equation}
and
\begin{equation}
\label{eqn:ExpSumBound2}
\max_{a\in\F_q^{*}}|S(a, H)| < q^{1/8} \cdot \max \bigg\{|H|^{(3-\delta)/4}, \frac{|H|^{25/33}}{q^{1/132}} \bigg\}. 
\end{equation}
Similarly, if $H$ satisfies condition~\eqref{eqn:Subgroupcnd2}, then the estimates~\eqref{eqn:ExpSumBound1} and~\eqref{eqn:ExpSumBound2} hold with $\delta = 1/33$.

Now, for $n\geq 1$, we fix $H$ to be the group of $n$th powers such that $|H| = (q -1)/n$. By~\eqref{eqn:subfieldintersection}, it is straightforward to confirm that if $n$ satisfies restrictions~\eqref{eqn:gausssumcnd1} or~\eqref{eqn:gausssumcnd2}, then $H$ satisfies restrictions~\eqref{eqn:Subgroupcnd1} or~\eqref{eqn:Subgroupcnd2} respectively. Finally, by~\eqref{eqn:EStoGS}, for any $a\in \F_q^*$, it follows
\begin{equation}
\label{gausssumidentity2}
|S_n(a)| \ll n|S(a, H)|,
\end{equation}
so that the estimates~\eqref{eqn:GSBound1} and~\eqref{eqn:GSBound2} follow from~\eqref{eqn:ExpSumBound1} and~\eqref{eqn:ExpSumBound2} respectively.
\end{proof}

\section{Proof of Theorem \ref{thm:IncSums}}
For $K\leq t= ord(g)$, let $X = \{g^i\}_{i=1}^{K}$ and $H = \langle g \rangle$, the subgroup of $\F_q^*$ generated by $g$. Clearly $X\subseteq H$, $|X/X| \approx |X|$ and $|H| = t$. Given a proper subfield $G$ of $\F_q$, with $|G| = p^{\nu}$, note that $|H\cap G| = \gcd(t, p^{\nu} -1)$. By Claim~\ref{claim:EquivalencyofConditions}, it follows that, if $t$ satisfies condition \eqref{eqn:IncSumCon}, then $X$ satisfies condition \eqref{eqn:energycnd2}. Hence, by Theorem~\ref{thm:AddEnergy}, we have a bound on $E_{+}(X)$. Then, observing the identity 
\begin{equation*}
\sum_{a \in \F_q}\Big|\sum_{x\in X}\psi_a(x)\Big|^{4} = qE_{+}(X),
\end{equation*}
we get estimate \eqref{eqn:IncPars}.

To prove \eqref{eqn:IncSum1} and \eqref{eqn:IncSum2}, we use a technique from \cite[Corollary~19]{RNRuShk}. Let
\[
\sigma(K) = \max_{1\leq J \leq K} \max_{a\in \F_q^*}\bigg|\sum_{j\leq J}\psi_a(g^j)\bigg|,
\]
and note that for any integer $J$, we have
\begin{equation}
\label{eqn:BilinearDifference}
\bigg|\sum_{k=1}^{K}\psi_a(g^{k}) - \frac{1}{J}\sum_{j=1}^{J} \sum_{k=1}^{K}\psi_a(g^{j+k}) \bigg|\leq 2\sigma(J).
\end{equation}
Now, let $Y = \{g^i\}_{i=1}^{J}$ for some arbitrary $J \leq K$. Then, since $Y\subseteq H$ and $|Y/Y| \approx |Y|$, we can use Theorem~\ref{thm:AddEnergy} to bound $E_{+}(Y)$. In particular, fixing $J = \left[K/4\right]$, by Theorem~\ref{thm:AddEnergy} and \eqref{eqn:BH1}, we have
\[
\frac{1}{J}\bigg|\sum_{j=1}^{J} \sum_{k=1}^{K}\psi_a(g^{j+k}) \bigg| \ll q^{1/4}K^{65/132} + q^{8/33}K^{67/132}.
\]
Consequently, by \eqref{eqn:BilinearDifference}, we have
\[
\sigma(K) \leq 2\sigma(K/4)+ O \Big(q^{1/4}K^{65/132} + q^{8/33}K^{67/132}\Big).
\]
This gives \eqref{eqn:IncSum1} by induction. The same argument can be repeated to prove \eqref{eqn:IncSum2} using \eqref{eqn:BH2}.

\section{Proof of Theorem \ref{thm:TrilinearSumsViaEnergy}}
Let 
\begin{equation*}
\bigg|\sum_{x\in X}\sum_{y\in Y} \sum_{z \in Z}\alpha_{x}\beta_{y}\gamma_{z}\psi(xyz)\bigg| = |X||Y||Z|\Delta.
\end{equation*}
We proceed to establish an upper bound on $\Delta$. An application of the triangle inequality gives 
\[
\sum_{x\in X}\bigg|\sum_{y\in Y} \sum_{z \in Z}\beta_{y}\gamma_{z}\psi(xyz)\bigg| \geq |X||Y||Z|\Delta.
\]
Then, by Lemma~\ref{lem:dyadic}, there exists a subset $A\subset X$ and a number $\delta_1 > 0$ with
\begin{equation}
\label{eqn:AandDelta1Size}
|A| \gg \frac{|X|\Delta}{\delta_{1}\log q}, \qquad \frac{\Delta}{\log q} \ll \delta_{1} \leq 1
\end{equation}
such that 
\begin{equation}
\label{eqn:FirstDyadic}
\delta_1 |Y||Z| \leq \bigg|\sum_{y\in Y}\sum_{z\in Z}\beta_{y}\gamma_{z}\psi(xyz)\bigg| \leq 2\delta_{1}|Y||Z| \qquad \textnormal{for any}\: x \in A.
\end{equation}
Now for some complex numbers $\theta_{x}$ with absolute value $1$ we have
\[
\sum_{x\in A}\sum_{y\in Y}\sum_{z\in Z}\theta_{x}\beta_{y}\gamma_{z}\psi(xyz) \geq \delta_{1}|A||Y||Z|.
\]
By the Cauchy-Schwarz inequality we obtain
\[
\sum_{a_1\in A}\sum_{a_2\in A}\bigg|\sum_{y\in Y}\sum_{z\in Z} \psi((a_1 -a_2)yz))\bigg|\geq \delta_{1}^{2}|A|^{2}|Y||Z|.
\]
By another application of Lemma~\ref{lem:dyadic}, there exists a set $G\subset A\times A$ and a number $\delta_2 >0$ with
\begin{equation}
\label{eqn:delta2}
|G|\gg\frac{|A|^{2}\delta_{1}^{2}}{\delta_{2}\log q}, \qquad \frac{\delta_{1}^{2}}{\log q} \ll \delta_{2} \leq 1
\end{equation}
such that for all pairs $(a_1, a_2) \in G$ we have
\[
\delta_{2} |Y||Z| \leq \bigg|\sum_{y\in Y}\sum_{z\in Z}\psi((a_1-a_2)yz)\bigg| \leq 2\delta_{2}|Y||Z|.
\]
Thus
\[
\delta_{2}^{2}|Y|^{2}|Z|^{2}|A\overset{G}{-}A| \leq \sum_{n\in \F_q}\bigg|\sum_{y\in Y}\sum_{z\in Z}\psi(nyz)\bigg|^{2} \leq q|Y||Z|^{2},
\]
which in turn implies that
\[
|A\overset{G}{-}A| \leq \frac{q}{\delta_{2}^{2}|Y|}.
\]
Using Lemma~\ref{BSG}, it follows that there exists a subset $A^{'} \subset A$, with
\begin{equation}
\label{eqn:TrilinearBSGSizeBound}
|A^{'}| \gg \frac{\delta_{1}^{2}|A|}{\delta_{2}\log q},
\end{equation}
such that
\begin{align}
\label{eqn:TrilinearBSGDifferenceBound}
|A^{'}- A^{'}| \ll \frac{q^{4}(\log q)^{5}}{|A|^{3}|Y|^{4}\delta_{1}^{10}\delta_{2}^{3}}.
\end{align}
For $\xi \in \F_q$, we define
\[
I(\xi) = |\{(x, y) \in A^{'} \times Y: xy = \xi\}|.
\]
Then, recalling \eqref{eqn:MECS}, we have
\begin{equation}
\label{eqn:A'YEnergy}
\sum_{\xi\in\F_{q}}I(\xi)^{2} = E_{\times}(A^{'}, Y) \leq |Y|^{3/2}E_{\times}(A^{'})^{1/2}.
\end{equation}
In particular, we used the trivial bound $E_{\times}(Y) \leq |Y|^{3}$. By~\eqref{eqn:FirstDyadic}, for some complex numbers $\theta_x$ with absolute value 1, we have
\[
\sum_{x\in A^{'}}\sum_{y\in Y}\sum_{z\in Z}\theta_{x}\beta_{y}\gamma_{z}\psi(xyz) \geq \delta_{1}|A^{'}||Y||Z|.
\]
We set
\[
\tilde{I}(\xi) = \sum_{\substack{(x, y) \in A^{'} \times Y \\ xy =\xi}} \theta_{x}\beta_{y},
\]
such that $|\tilde{I}(\xi)| \leq I(\xi)$. Then, by Lemma~\ref{lem:ClassicalBilinearSum} and~\eqref{eqn:A'YEnergy}, it follows
\begin{align*}
\delta_1 |A^{'}||Y||Z| &\leq \bigg|\sum_{\xi\in\F_{q}}\sum_{z\in Z}\tilde{I}(\xi)\gamma_{z}\psi(\xi z)\bigg| \\
& \leq q^{1/2} \bigg(\sum_{\xi\in\F_{q}}|\tilde{I}(\xi)|^{2}\bigg)^{1/2}|Z|^{1/2} \\
&\leq q^{1/2}|Z|^{1/2}|Y|^{3/4}E_{\times}(A^{'})^{1/4}.
\end{align*}
Under restriction~\eqref{eqn:trilinearantifield}, we may apply Lemma~\ref{lem:energybound} to the set $A^{'}$ to bound $E_{\times}(A^{'})$. Then, based on the first term of \eqref{eqn:MEB}, we have
\[
\delta_{1}^{80}|A^{'}|^{80}|Y|^{20}|Z|^{40}q^{-40} \ll (\log|A^{'}|)^{20}|A^{'}-A^{'}|^{35}|A^{'}|^{20}.
\]
By~\eqref{eqn:TrilinearBSGSizeBound} and~\eqref{eqn:TrilinearBSGDifferenceBound} we get
\[
\delta_{1}^{550}\delta_{2}^{45}|A|^{165}|Y|^{160}|Z|^{40}\ll q^{180}(\log q)^{O(1)}.
\]
We eliminate $\delta_2$ by~\eqref{eqn:delta2}, then use~\eqref{eqn:AandDelta1Size} to replace $|A|$ and $\delta_1$, to get
\begin{equation}
\label{eqn:T1stTerm}
|X|^{33}|Y|^{32}|Z|^{8}\Delta^{128} \ll q^{36+o(1)}.
\end{equation}
Next, based on the second term of \eqref{eqn:MEB}, we have
\[
\delta_{1}^{80}|A^{'}|^{80}|Y|^{20}|Z|^{40}q^{-40} \ll (\log|A^{'}|)^{20}|A^{'}- A^{'}|^{24}|A^{'}|^{32}.
\]
We use~\eqref{eqn:TrilinearBSGSizeBound} and~\eqref{eqn:TrilinearBSGDifferenceBound}, to get
\[
\delta_{1}^{416}\delta_{2}^{24}|A|^{120}|Y|^{116}|Z|^{40}\ll q^{136}(\log q)^{O(1)}.
\]
We replace $\delta_2$ using \eqref{eqn:delta2}, and then replace $|A|$ and $\delta_1$ by~\eqref{eqn:AandDelta1Size}. This gives
\begin{equation}
\label{eqn:T2ndTerm}
|X|^{30}|Y|^{29}|Z|^{10}\Delta^{116} \ll q^{34+o(1)}.
\end{equation}
Following the same process as above, based on the third term of \eqref{eqn:MEB}, we get
\begin{equation*}
|X|^{31}|Y|^{32}|Z|^{8}\Delta^{128} \ll q^{35+o(1)}.
\end{equation*}
Putting it all together, we obtain the required bound on $\Delta$, which in turn concludes the proof of estimate \eqref{eqn:TrilinearB}.

\section{Proof of Corollary \ref{cor:TrilinearEqSizes}}
Note that the second term of \eqref{eqn:TrilinearB} dominates the first term if 
\begin{equation}
\label{eqn:trilinearRange1v3}
|Z||X|^{3/88} \ll q^{1/2+ o(1)}.
\end{equation}
Then, under the consideration that $|X| = |Y| = |Z|$, in the range implied by~\eqref{eqn:trilinearRange1v3}, trivially, we have
\begin{align*}
\bigg|\sum_{x\in X}\sum_{y\in Y} \sum_{z \in Z}\alpha_{x}\beta_{y}\gamma_{z}\psi(xyz)\bigg|  \leq |X|^{3} &< q^{9/32 + o(1)} |X|^{311/128}\bigg(\frac{|X|^{73/128}}{q^{9/32}}\bigg) \\ &\ll q^{9/32 + o(1)} |X|^{311/128}\bigg(\frac{1}{|X|^{1/88}}\bigg)
\\ &< q^{9/32 + o(1)} |X|^{311/128}.
\end{align*}
Thus the second term of \eqref{eqn:TrilinearB} can be eliminated, as required.

\section{Proof of Theorem \ref{thm:GaussSumViaTrilinearSum}}
Let $H$ denote a multiplicative subgroup of $\F_q^*$. Suppose that for all proper subfields $G$ of $\F_q$, we have $|H\cap G| \ll |G|^{1/2}.$ Then, by Claim~\ref{claim:EquivalencyofConditions}, $H$ also satisfies condition~\eqref{eqn:trilinearantifield}. Hence, by Theorem~\ref{thm:TrilinearSumsViaEnergy}, for $a\in \F_q^*$, we have
\begin{align}
\label{eqn:ESBoundViaTrilinear}
 \bigg|\sum_{h\in H}\psi_a(h)\bigg| &= \frac{1}{|H|^{2}}\bigg|\sum_{h_1\in aH} \sum_{h_2\in H}\sum_{h_3\in H}\psi(h_1h_2 h_3)\bigg| \\[1em]
\nonumber& \ll q^{9/32 + o(1)}\Big(|H|^{55/128} + q^{-1/128}|H|^{57/128}\Big).
\end{align}
We fix $H$ to be the group of $n$th powers, such that if $n$ satisfies restriction~\eqref{eqn:gausssumcnd2}, then $H$ satisfies restriction~\eqref{eqn:Subgroupcnd2}. The desired estimate on $S_n(a)$ follows form~\eqref{gausssumidentity2} and~\eqref{eqn:ESBoundViaTrilinear}.

\section*{Acknowledgement}
The author would like to thank Antal Balog and Igor Shparlinski for their helpful comments on the manuscript.

\end{document}